\documentclass[AMS, reqno]{amsart}

\usepackage{amsmath, amsthm, amscd, amsfonts, amssymb, graphicx, color, textcomp,  mathrsfs}
\usepackage[english]{babel}

\newtheorem{theorem}{Theorem}[section]
\newtheorem{lemma}[theorem]{Lemma}
\newtheorem{properties}[theorem]{Properties}
\newtheorem{proposition}[theorem]{Proposition}
\newtheorem{corollary}[theorem]{Corollary}
\theoremstyle{definition}
\newtheorem{definition}[theorem]{Definition}

\theoremstyle{remark}
\newtheorem{remark}[theorem]{Remark}
 \numberwithin{equation}{section}
\def\R{\mathbb{R}}

\def\C{\mathbb{C}}

\def\ie{i.e.\ }

\def\eps{\varepsilon}
\def\supp{{\mathrm{supp}\,}}
\def\eg{{\it e.g.\ }}
\def\d{\,{\mathrm{d}}}

\newcommand{\norm}[1]{{\left\|{#1}\right\|}}

\newcommand{\abs}[1]{{\left|{#1}\right|}}
\newcommand{\scal}[1]{{\left\langle{#1}\right\rangle}}

\def\L{L^2_{\alpha}(\R_+)}
\def\l{L^2_{\alpha}}
\def\r{\mathbf{R_+^2}}
\def\u{\mathbf{U}}
\def\D{\mathcal{D}^\alpha}
\def\H{\mathcal{H}_{\alpha}}
\def\w{\mathbf{W}}
\def\W{\mathbf{W}^{\alpha}}
\def\m{\mu_{\alpha}}
\def\n{\nu_{\alpha}}
\def\e{{\mathbf{E}}}
\def\ss{{\mathcal S}}
\def\ff{{\mathcal F}}

\date{\today}

\begin{document}

\title[Logarithmic uncertainty principles for the HWT]{Logarithmic uncertainty principles for the Hankel wavelet transform}

\author[S. Ghobber]{Saifallah Ghobber$^{1,2}$}

\address{$^{1}$ Department of Mathematics and Statistics, College of Science, King Faisal University, PO.Box: 400  Al-Ahsa  31982,   Saudi Arabia}
 \email{\textcolor[rgb]{0.00,0.00,0.84}{sghobber@kfu.edu.sa}}

 \address{$^{2}$ LR11ES11 Analyse Math\'{e}matiques et Applications, Facult\'{e} des Sciences de Tunis, Universit\'{e} de Tunis El Manar, 2092 Tunis, Tunisia}
  \email{\textcolor[rgb]{0.00,0.00,0.84}{saifallah.ghobber@fst.utm.tn}}


\begin{abstract}
The aim of this paper is to prove a logarithmic and a Hirschman-Beckner entropic uncertainty principles for the Hankel wavelet transform. Then  we derive a general form of Heisenberg-type uncertainty inequality for this transformation.
\end{abstract}


\maketitle
%


\section{Introduction}
Let $d \ge1$ be the dimension, and let us denote by $\scal{\cdot,\cdot}$ the scalar product and by $|\cdot|$ the Euclidean norm on $\R^d$.
Define also the classical translation and dilation operators on $ L^2(\R^d)$ by
\begin{equation}
\tau_xg(t)=g(t-x)\quad  \mathrm{and} \quad \delta_ag(t)= a^{d/2}f\left( at \right),\qquad  a\in (0,\infty),\quad x,t\in\R^d.
\end{equation}
Then the Fourier transform is defined  for a function $g\in L^1(\R^d)\cap  L^2(\R^d)$   by:
\begin{equation}
\ff(g)(\xi)= \widehat{g}(\xi)=(2\pi)^{-d/2}\int_{\R^d}g(x) e^{-i\scal{x,\xi}}\d x,
\end{equation}
and it is extended from $L^1(\R^d)\cap L^2(\R^d)$ to $L^2(\R^d)$ in the usual way.
Moreover, if we take $\phi\in L^2(\R^d)$   an admissible wavelet, that satisfies the following admissibility condition,
  \begin{equation}
0<A_{\phi}:= \int_{0}^\infty | \widehat{ \phi}(a \xi)|^2\frac{\d a}{a}<\infty,\qquad \forall\xi\in\R^d\backslash\{0\},
\end{equation}
 then the continuous wavelet transform of a signal $f\in L^2(\R^d)$ is defined by (see \eg \cite{grochenig}),
\begin{equation}
\w_\phi g(a,x)=(2\pi)^{-d/2} \int_{\R^d} g(t) \overline{\phi_{a,x}(t)}\d t,   \qquad (a,x)\in\u,
\end{equation}
where $\u= (0,\infty)\times \R^d$ is the affine group, and $\phi_{a,x} $ is the wavelet atom defined, by
\begin{equation}
\phi_{a,x} = A_{\psi}^{-1/2} \tau_x\delta_a \phi .
\end{equation}
Wavelet atoms (which are the dilation with a scale parameter $a$ and the translation by the position (or time) parameter $x$ of admissible wavelets) has their energy well localized in position, while their Fourier transform is mostly concentrated in a limited frequency band.

\medskip

Now,  if $g(x)=f(|x|)$ is a radial function on $\R^d$, then $\widehat{g}(\xi)=\mathcal{H}_{d/2-1}(f)(|\xi|)$, where for $\alpha\ge-1/2$, $\H$ is the   Hankel  transform (also known as the Fourier-Bessel transform) defined by (see \eg \cite{Stempak}):
\begin{equation}
\H(f) (\xi)=\int_{\R_+} f(x)j_\alpha(x\xi)\d\mu_\alpha (x),\ \ \ \  \xi\in\R_+=[0,+\infty).
\end{equation}
In particular  $\mu_\alpha$ is the weight measure defined by $\d\mu_\alpha (x)=\frac{x^{2\alpha+1}}{2^{\alpha}\Gamma(\alpha+1)}\,\d x$  and $j_\alpha$ (see \eg \cite{Stempak, watson}) is the spherical Bessel function given by:
\begin{equation}
j_\alpha(x) =2^\alpha\Gamma(\alpha+1)\frac{J_\alpha (x)}{x^\alpha}
:=\Gamma(\alpha+1)\sum_{n=0}^{\infty}\frac{(-1)^n}{n!\Gamma(n+\alpha+1)}\left(\frac{x}{2}\right)^{2n},
\end{equation}
where $J_\alpha$ is the Bessel function of the first kind and $\Gamma$ is the gamma function.

\medskip

Throughout this paper, $\alpha$   will be a real number such that  $\alpha>-1/2$, since for $\alpha=-1/2$, we have  $\mu_{-1/2}$ is the Lebesgue measure and $ \mathcal{H}_{-1/2}$ is the  Fourier-cosine transform,
which is the Fourier transform $\ff$ restricted to even functions on $\R$.

\medskip

For $\alpha>-1/2$, let us recall the  Poisson representation formula (see \eg \cite[(1.71.6), p. 15]{S}):
\begin{equation}
j_\alpha(x)=\frac{\Gamma(\alpha+1)}{\Gamma\left(\alpha+\frac{1}{2}\right)\Gamma\left(\frac{1}{2}\right)}
\int_{-1}^1 (1-s^2)^{\alpha-1/2}\cos(sx) \d s.
\end{equation}
Therefore, $j_\alpha$ is bounded with $|j_\alpha(x)|\leq j_\alpha(0)=1$.
As a consequence,
\begin{equation}\label{eq:L1infty}
\norm{\H(f)}_\infty\leq \norm{f}_{1,\m},
\end{equation}
where $\norm{.}_\infty$ is the usual essential supremum norm on the space of essentially bounded functions   $L^\infty_\alpha(\R_+)$,
 and  for $1\leq p < \infty$, we denote by $L^p_\alpha(\R_+)=L^p(\R_+,\m)$ the Banach space consisting of measurable functions $f$ on $\R_+$ equipped with the norms:
\begin{equation}
\norm{f}_{p,\m}=\left( \int_{0}^\infty \abs{f(x)}^p\,\d \mu_\alpha(x)  \right)^{1/p}.
\end{equation}
It is also well-known (see   \cite{Stempak})  that the Hankel transform extends to an isometry on $ L^2_\alpha(\R_+)$:
\begin{equation}\label{parseval}
 \|\H(f) \|_{2,\m}= \|f \|_{2,\m}.
\end{equation}


Uncertainty principles in Fourier analysis set a limit to the possible concentration of a function and its Hankel transform in the time-frequency domain.
The most familiar form is the Heisenberg-type inequalities, in which concentration is measured by dispersions (see \cite{bowie, rosler, JAT}):  For all $f\in\L$, 
 \begin{equation}\label{eq:heisHTint}
 \|x f \|_{2,\m} \, \|\xi\,\H (f) \|_{2,\m}\geq  (\alpha+1)\|f \|_{2,\m}^2, 
\end{equation}
with equality, if and only if $f$ is a multiple of a suitable  Gaussian function.

\medskip

A little less known principles  consist of logarithmic uncertainty principles. In particular, for any even function $f$ in  the Schwartz space $ \ss_e(\R)$  we have (see \cite[Theorem 3.10]{omri}),
\begin{equation}\label{clupint}
 \int_{0}^\infty  \ln(x) \:|f(x)|^2 \d\m(x)+  \int_{0}^\infty  \ln(\xi) \: |\H(f)(\xi)|^2 \d\m(\xi)\ge \left(\ln(2)+\frac{\Gamma'(\frac{\alpha+1}{2})}{\Gamma(\frac{\alpha+1}{2})}\right)  \|f\|^2_{2,\m}.
\end{equation}
The proof of the last inequality is based on a Pitt-type inequality for the Hankel transform (see \cite[Theorem 3.9]{omri}), and from which we derive the following Heisenberg-type uncertainty inequality for the Hankel transform: For all even function $f$ in  the Schwartz space $ \ss_e(\R)$,
\begin{equation}\label{eqHcwthankelnewint}
 \left\| x\,f \right\|_{2,\m}  \left\| \xi  \, \H(f) \right\|_{2,\m} \ge  2\exp\left( \frac{\Gamma'(\frac{\alpha+1}{2})}{\Gamma(\frac{\alpha+1}{2})} \right) \,\|f\|^2_{2,\m},
\end{equation}
where \begin{equation}
 2\exp\left( \frac{\Gamma'(\frac{\alpha+1}{2})}{\Gamma(\frac{\alpha+1}{2})} \right)  \approx (\alpha+1), \qquad \mathrm{for} \quad \alpha\gg 1,
\end{equation}
which is the optimal constant in the sharp Heisenberg's Inequality \eqref{eq:heisHTint}.

Moreover, by using the sharp Hausdorff-Young inequality for the Hankel transform, the author in \cite[Theorem 3.3]{cubo}, proved the following entropic uncertainty inequality for the Hankel transform (in which concentration is measured by entropy, and which implies also Inequality \eqref{eq:heisHTint}), that is, for all nonzero $f\in \L$,
\begin{eqnarray}\label{ineqecHankelint}
  &- & \int_0^\infty |f(x)|^2 \,\ln\big(|f(x)|^2\big)\d \m(x)-\int_0^\infty |\H(f)(\xi)|^2 \,\ln\big(|\H(f)(\xi)|^2\big)\d\m(\xi) \nonumber\\
  & \ge& \left( (2\alpha+2)\ln\left(\frac{e}{2}\right)-2\ln\left(\|f \|_{2,\m}^2\right)\right)\|f \|_{2,\m}^2.
\end{eqnarray}

\medskip

As for  the Fourier transform, if we take a radial  admissible wavelet $\phi(t)= \psi (|t|)$ , then for all radial function $g(t)=f(|t|)$, its continuous  wavelet transform $\w_\phi g$ coincides with the Hankel wavelet transform $\W_{ \psi } f$ (with $\alpha=d/2-1$), defined by
 \begin{equation} \label{defCWTint}
\W_{ \psi } f(a,x)=c_{\psi}^{-1/2}\int_0^{ \infty} f(t)\,\overline{\tau_x^{\alpha}\left(\D_a\psi\right)}(t)\d\m(t),\qquad (a,x)\in\r=(0,\infty)\times[0,\infty),
 \end{equation}
 where 
$c_{\psi}$ is the admissibility condition given in \eqref{admicondintro}, and $\tau_x^{\alpha}$, $\D_a$ are the Hankel translation and the dilation operators given in \eqref{tau0}, \eqref{dt0} respectively.

\medskip

Wavelet theory is often seen as a relatively new concept for time-frequency analysis, rather than the Fourier, the Hankel and the windowed Fourier  transforms, because of  its advantage on better locality in the time-scale variations of a signal. However, the uncertainty principles   set a limit to the
maximal time-frequency or time-scale resolutions. These  uncertainty principles for the continuous wavelet transform, can be found in \cite{dahlke, Ghobber,HL, Medina, sing, WZK}, and those  for the Hankel wavelet transform can be found in \cite{cyrine, bhn}.

\medskip

The aim of this paper is to extend the uncertainty principles \eqref{clupint} and \eqref{ineqecHankelint} for the Hankel wavelet transform. In particular we prove the following Pitt-type inequality: For all $0\le \beta< \alpha+1$, and all function $f\in\ss_e(\R)$ such that $\W_{\psi}f(a,\cdot)\in \ss_e(\R)$, we have
 \begin{equation}  \label{pittcwtint}
    \int_{0}^\infty \int_{0}^\infty  a^{-2\beta} \, |\W_{\psi}f(a,x)|^2   \d\n(a,x)
 \le  C_{\alpha,\beta}(\psi)   \int_{0}^\infty \int_{0}^\infty x^{2\beta} \, |\W_{\psi}f(a,x)|^2 \d\n(a,x),
 \end{equation}
 where $ \d\n(a,x)= a^{2\alpha+1}\d a\d\m(x)$, and the constant $C_{\alpha,\beta}(\psi)$ is given in \eqref{pittcwtconstant}. This inequality allows to prove the following logarithmic uncertainty principle for the Hankel wavelet transform: For all function $f\in\ss_e(\R)$ such that $\W_{\psi}f(a,\cdot)\in \ss_e(\R)$, we have
 \begin{equation}  \label{lncwtint}
    \int_{\r}  \ln(a) \:\abs{\W_{\psi}f(a,x)}^2 \d\n(a,x)+    \int_{\r} \ln(x) \:\abs{\W_{\psi}f(a,x)}^2 \d\n(a,x)  \ge C_{\alpha}(\psi)  \|f\|^2_{2,\m},
 \end{equation}
 where $\r=(0,\infty)\times \R_+$, and the constant $C_{\alpha}(\psi) $ is given in \eqref{lncwtconstant}. Consequently, we derive the following Heisenberg-type uncertainty inequality:
 For all $f\in \ss_e(\R)$ such that $\W_{\psi}f(a,\cdot)\in \ss_e(\R)$, we have
\begin{equation}\label{eqHcwtint}
 \left\| a  \, \W_{\psi}f \right\|_{2,\n}  \left\| x  \, \W_{\psi}f \right\|_{2,\n} \ge   e^{C_\alpha(\psi)}  \,\|f\|^2_{2,\m},
\end{equation}
where $\left\|\cdot\right\|_{p,\n}$, ($p\ge1$) are the usual norms in the Banach spaces $L^p_\alpha(\r)=L^p(\r,\n)$ given by,
\begin{equation}
\|F\|_{p,\n}^p=\int_0^\infty\int_0^\infty |F(a,x)|^p\d\n(a,x),\qquad F\in L_\alpha^{p}(\r).
\end{equation}

\medskip

Finally, we prove the following Hirschman-Beckner entropic uncertainty inequality for the Hankel wavelet transform, that is,
for any    admissible wavelet $\psi$  such that $\|\psi \|_{2,\m}^2\le c_\psi$, we have for all  nonzero  function $f\in \L$,
\begin{equation} \label{ineqEint}
-\int_\r |\W_\psi f(a,x)|^2\ln \left(|\W_\psi f(a,x)|^2\right)\d\n(a,x)
\ge \|f \|_{2,\m}^2   \ln \left(\frac{c_\psi}{\|\psi\|_{2,\m}^2\|f \|_{2,\m}^2}\right).
\end{equation}
The last inequality implies a general form of Heisenberg-type uncertainty inequality for functions in $\L$, that is,
for all $s,\beta>0$, and any admissible  wavelet  $\psi$  such that $\|\psi \|_{2,\m}^2\le c_\psi$,   there exists a positive constant $C(s,\alpha,\beta)$ such that, for all nonzero function $f\in \L$,
\begin{equation} \label{heiproint}
\left\|a^{ s} \, \W_\psi f  \right\|_{2,\n}^\beta  \left\|x^\beta \,\W_\psi f \right \|_{2,\n}^s  \ge C(s,\alpha,\beta) \, \|f \|_{2,\m}^{s+\beta},
\end{equation}
where the constant $C(s,\alpha,\beta)$ is given in \eqref{heiproconstant}. This inequality improve a result proved in \cite{cyrine}, in which the constant involves the Mellin transform.

\medskip

As a side result, we prove that for any admissible  wavelet  $\psi$, and any function $f\in \L$, its Hankel wavelet transform $\W_\psi f$ belongs to $L^p_\alpha(\r)$, $p\ge2$, satisfying the following Lieb-type inequality
   \begin{equation}\label{eqpsiphinormint}
 \left\|\W_\psi  f\right\|_{p,\n}\le \left(\frac{\|\psi \|_{2,\m}^2}{ c_\psi  }\right)^{ \frac{1}{2}- \frac{1}{p}}   \|f\|_{2,\m}  .
\end{equation}
In particular if a nonzero function $f  \in \L$ is $(\eps,\psi)$-time-scale-concentrated in a subset $\Sigma\subset\r$ of finite measure, \ie
\begin{equation}\label{eqdefepsint}
 \int_{\r\backslash\Sigma } |\W_\psi f(a,x)|^2\d \n(a,x)\le \eps \|f\|_{2,\m}^{2},
\end{equation}
where $0\le\eps<1$, then $\Sigma$ satisfies  for every $p>2$,
\begin{equation}\label{eqliebthint}
\n(\Sigma)\ge  \frac{c_\psi}{\|\psi \|_{2,\m}^2} \, (1-\eps)^{\frac{p}{p-2}} .
\end{equation}
Consequently, if $\eps = 0$, then \eqref{eqdefepsint} implies that the Hankel wavelet transform is concentrated in $\Sigma$ and its support $\Sigma$ satisfies
\begin{equation}
 \n(\supp \W_\psi  f)\ge \frac{c_\psi}{\|\psi \|_{2,\m}^2},
\end{equation}
which means that, the support  of the Hankel wavelet transform of a nonzero function $f  \in \L$ cannot be too small.

\medskip

 The remainder of this paper is arranged as follows, in the second section we recall some useful harmonic analysis results associated with the  Hankel   and the Hankel wavelet transforms. Section 3 is devoted to the study of an uncertainty principle for the Hankel wavelet transform in subsets of small measures. In Section 4 we prove some logarithmic uncertainty principles for the Hankel wavelet transform, and derive some Heisenberg-type uncertainty inequalities.

\section{Preliminaries}
\subsection{Notation}

For a measurable subset $E$, we will write $E^c$ for its  complement, and we will denote by $\ss_e(\R)$ the Schwartz space, constituted by the even infinitely differentiable functions on the real line, rapidly decreasing together with all their derivatives.

\subsection{The Hankel transform}
In this subsection, we will recall some harmonic analysis results related to the Hankel transform that we shall use later (see \eg \cite{Stempak, watson}).
We denote by $\ell_{\alpha}$  the Bessel operator defined on $(0,\infty) $ by
\begin{equation}
\ell_{\alpha}(u)=u''+\frac{2\alpha+1}{r}u'.
\end{equation}
For all $\lambda\in\C$, the following system
\begin{equation}
\left\{\begin{array}{l}
\ell_{\alpha}(u)=-\lambda^2u,\\
u(0)=1,\; u'(0)=0;
\end{array}\right.
\end{equation}
admits a unique solution given by the modified Bessel function $x\mapsto j_{\alpha}(x\lambda)$, where
\begin{equation}
j_{\alpha}(z)=\frac{2^{\alpha}\Gamma(\alpha+1)}{z^{\alpha}}J_{\alpha}(z)=\Gamma(\alpha+1)\sum_{n=0}^{+\infty}\frac{(-1)^{n}}{n!\Gamma(\alpha+n+1)}
\left(\frac{z}{2}\right)^{2n}\ z\ \in \mathbb{C}.
\end{equation}

\medskip

The Hankel translation operator $\tau_x^{\alpha}$, $x\in\R_+$ is defined  by
\begin{equation}\label{tau0}
   \tau_x^{\alpha}(f)(y)=
 \frac{\Gamma(\alpha+1)}{\Gamma(\frac{1}{2})\Gamma(\alpha+\frac{1}{2})}\int_{0}^{\pi}f\left(\sqrt{x^2+y^2+2xy\cos(\theta)}\right)\sin(\theta)^{2\alpha}
 \d\theta ,
  \end{equation}
whenever the integral in the right-hand side is well defined. Then
\begin{equation}
 \tau_x^{\alpha}\big(j_{\alpha}(\lambda\cdot)\big)(y)=j_{\alpha}(\lambda  x)j_{\alpha}(\lambda y),\qquad\,x,y, \lambda\geq0.
\end{equation}
Moreover, for all $x,y>0$,
\begin{equation}
\label{tau integral}
 \tau_x^{\alpha}(f)(y)=\int_0^{ \infty}f(t)K_{\alpha}(t,x,y)\d\m(t),
\end{equation}
where $K_{\alpha}$ is the kernel given by
\begin{equation}
\label{W}
   K_{\alpha}(t,x,y)=\left\{\begin{array}{ll}
\frac{\Gamma^2(\alpha+1)}{\sqrt{\pi}2^{\alpha-1}\Gamma(\alpha+\frac{1}{2})}\frac{\left([(x+y)^2-t^2][t^2-(x-y)^2]\right)^{\alpha-1/2}}{(xyt)^{2\alpha}},&\mathrm{if}\;|x-y|<t<x+y,\\ \\
 0,& \mathrm{otherwise}.
        \end{array}
\right.
  \end{equation}
The kernel $K_{\alpha}$ is symmetric in the variables $t,x,y$, and satisfies
\begin{equation}
\label{1}
 \int_0^{ \infty}K_{\alpha}(t,x,y)\d\m(t)=1.
\end{equation}
From H\"older's inequality, Relations \eqref{tau integral} and \eqref{1}
we have, for every $f\in L^p_{\alpha}(\R_+)$, $p\in[1, \infty]$, the function $\tau_x^{\alpha}(f)\in L^p_{\alpha}(\R_+)$ and
\begin{equation}\label{tau}
 \|\tau_x^{\alpha}(f)\|_{p,\m}\leq\|f\|_{p,\m}.
\end{equation}
Moreover, for $f\in L^1_{\alpha}(\R_+)$, we have for all $x\in\R_+$,
\begin{equation}
 \int_0^{\infty}\tau_x^{\alpha}(f)(y)\d\m(y)=\int_0^{\infty}f(y)\d\m(y),
\end{equation}
and for every $f\in L^p_{\alpha}(\R_+)$, $p=1,2$,
  \begin{equation}
   \H(\tau_x^{\alpha}f)(\lambda)=j_{\alpha}(x\lambda)\H(f)(\lambda),\qquad\lambda\in\R_+.
  \end{equation}

\medskip

The convolution product of $f,g\in L^1_{\alpha}(\R_+)$ is defined by
\begin{equation}
 f\ast_{\alpha}g(x)=\int_0^{ \infty}\tau_x^{\alpha}(f)(y)g(y)\d\m(y),
\end{equation}
which is commutative and associative in $L^1_{\alpha}(\R_+)$. Then, for every $1\le p,q,r\le\infty$, such  that $1/p+1/q=1+1/r$, we have $ f\ast_{\alpha}g\in L^r_{\alpha}(\R_+)$, with
$$
 \|f\ast_{\alpha}g\|_{r,\m}\leq\|f\|_{p,\m}\|g\|_{q,\m}.
$$
In particular, for every $f\in L^1_{\alpha}(\R_+)$ and $g\in L^p_{\alpha}(\R_+)$, $p=1,2$, the function $f\ast_{\alpha}g$ is in $L^p_{\alpha}(\R_+)$, $p=1,2$ and
  \begin{equation}
   \H(f\ast_{\alpha}g)=\H(f)\H(g).
  \end{equation}
Moreover,  for every $f,g\in\L$, the function $f\ast_{\alpha}g\in\L$ if and only if $\H(f)\H(g)\in\L$, and then
\begin{equation}
 \H(f\ast_{\alpha}g)=\H(f)\H(g).
\end{equation}

 \subsection{The   Hankel wavelet transform}
Now, we will recall some preliminaries results of the   Hankel wavelet transform (HWT for short), which is first introduced in \cite{trimeche}, (see also \cite{cyrine, bhn,  peng, pathak}).

\medskip

The dilation operator $\D_a$ of a function $f\in\L$  is defined  by
 \begin{equation}\label{dt0}
  \D_a f(x)=a^{\alpha+1}f(ax),\qquad a>0.
 \end{equation}
 It satisfies the following properties.
\begin{properties}\
 \begin{enumerate}
  \item For all $f\in\L$  we have
  \begin{equation}\label{dt}
   \|\D_a f\|_{2,\m}=\|f\|_{2,\m},\qquad \H\left(\D_a f\right)=\D_{\frac{1}{a}}\H(f).
  \end{equation}\label{dtscal}
  \item For all $f,g\in\L$, we have
  \begin{equation}
   \langle\D_a f,g\rangle_{\m}=\langle f,\D_{\frac{1}{a}}g\rangle_{\m},
  \end{equation}
     where $\scal{\cdot,\cdot}_{\m}$ is the usual inner product on the Hilbert space $\L$.
\item For all $x\geq0$, we have
\begin{equation}\label{dttau}
 \D_a\tau_x^{\alpha}=\tau_{\frac{x}{a}}\D_a.
\end{equation}
 \end{enumerate}
\end{properties}
\begin{definition}
 A nonzero function $\psi\in L^1_\alpha(\R_+)\cap L^\infty_\alpha(\R_+)$ is said to be an admissible   wavelet, if it satisfies,
 \begin{equation}\label{admicondintro}
  0<c_\psi=\int_0^{ \infty}|\H(\psi)(a)|^2\frac{\d a}{a}< \infty.
 \end{equation}
\end{definition}
Moreover, if  $\psi$ is an  admissible   wavelet, then $\psi   \in L^p_\alpha(\R_+)$, for all $1\le p\le \infty$, since for $p\in(1,\infty)$, we have
 \begin{equation}
\|\psi \|_{p,\m}^p=\int_{\R_+} |\psi(x)| |\psi(x)|^{p-1}\d\m(x)\le \|\psi \|_{\infty}^{p-1} \|\psi \|_{1,\m}.
 \end{equation}
\begin{definition}
 Let $\psi$ be an admissible  wavelet. Then the Hankel wavelet transform $\W_\psi$ is defined on $\L$ by
 \begin{equation}\label{defCWT}
  \W_\psi f(a,x)=\int_0^{ \infty}f(t)\overline{\psi_{a,x}^{\alpha}(t)}\d\m(t),\qquad (a,x)\in\r,
 \end{equation}
 where, $\psi_{a,x}^{\alpha}=c_\psi^{-1/2} \tau_x^{\alpha}\left(\D_a\psi\right)$.
\end{definition}
Equality \eqref{defCWT} can be also   written as
\begin{equation}
 \W_\psi f(a,x) =c_\psi^{-1/2}f\ast_{\alpha}\D_a(\overline{\psi})(x) =\scal{ f,\psi_{a,x}}_{\m}.
 \end{equation}


Moreover, the Hankel wavelet transform satisfies the following properties.
\begin{properties}
  Let $\psi$ be an admissible   wavelet. Then for every $f\in\L$, we have:
  \begin{enumerate}
   \item The function $\W_\psi f $  is bounded, with
 \begin{equation}\label{norminfty}
  \|\W_\psi f\|_{\infty}\leq c_\psi^{-1/2}\|f\|_{2,\m}\|\psi\|_{2,\m},
 \end{equation}
 where  $\|\cdot\|_{\infty}$ is the essential supremum norm in the space $L_\alpha^{\infty}(\r)$.
 \item The function $\W_\psi f $  belongs to $L_\alpha^{2}(\r)$, and satisfies the following Plancherel formula,
   \begin{equation} \label{plancherel}
  \|\W_\psi f\|_{2,\n}= \|f\|_{2,\m}.
 \end{equation}
\item For every $(a,x)\in\r$,
\begin{equation}\label{dtt}
 \W_\psi(\D_\lambda f)(a,x)=\W_\psi f\left(\frac{a}{\lambda}, \lambda x\right),\qquad \lambda>0.
\end{equation}
\item For every $(a,x)\in\r$, we have
\begin{equation}\label{cwti}
  \H\left(\W_\psi f(a,\cdot)\right)(\xi)=\frac{c_\psi^{-1/2}}{a^{\alpha+1}}\H(f)\overline{\H(\psi)}\left(\frac{\xi}{a}\right).
 \end{equation}
\end{enumerate}
\end{properties}

\section{Uncertainty principles for the HWT in subsets of small measures }
In this section $\Sigma$ will be a  subset in $\r$ of finite measure $0<\n(\Sigma)< \infty$,  and $\psi$  will be   an admissible wavelet. The aim of this section is to prove an uncertainty results limiting the concentration of the HWT in subsets of small measures (see \eg \cite{DS} and \cite[Theorem 3.3.3]{grochenig} for related results for the Fourier and the windowed Fourier transforms).
\begin{definition} Let $0\le\eps<1$. Then,  
  a  nonzero   function $f  \in \L$ is $(\eps,\psi)$-time-scale-concentrated in  $\Sigma$  if,
\begin{equation}\label{eqdefeps}
 \int_{\Sigma^c} |\W_\psi f(a,x)|^2\d \n(a,x)\le \eps \|f\|_{2,\m}^{2}.
\end{equation}
    \end{definition}
 Notice also that, if  $\eps = 0$ in  \eqref{eqdefeps}, then $\Sigma$ will be  the exact support of $\W_\psi f $ (denoted by $\supp \W_\psi f$), so that   Inequality \eqref{eqdefeps} with  $0<\eps<1$,  means that $\W_\psi f $ is ``practically zero'' outside $\Sigma$, and then $\Sigma$ may be considered as the ``essential'' support of $\W_\psi f $.

  We denote by $\l(\eps, \Sigma, \psi)$ the set of all functions in $f\in \L$ that are $(\eps,\psi)$-time-scale-concentrated in $\Sigma$, where $\eps\in(0,1)$.
 Then we have the following Donoho-Stark type uncertainty relation.  (see \cite[Proposition 3.1]{cyrine}).

\begin{proposition}\
If   $f$ is in $\l(\eps,\Sigma,\psi)$, then the essentiel support of its HWT satisfies
\begin{equation}\label{eqds}
\n(\Sigma)\ge  \frac{c_\psi}{\|\psi \|_{2,\m}^2} \; (1-\eps) .
\end{equation}
\end{proposition}

\begin{proof}
Let $f\in\l(\eps,\Sigma,\psi)$ such that $\|f \|_{2,\m}=1$.
By \eqref{norminfty}, we have for all $(a,x)\in\r$,
  $$
  |\W_\psi f(a,x)|  \le   \frac{\|\psi \|_{2,\m}}{\sqrt{c_\psi}},
  $$
then
$$
 \int_{\Sigma}\abs{\W_\psi  f (a,x)}^{2} \d\n(a,b)\le \|\W_\psi  f\|_\infty^2 \n(\Sigma) \le \n(\Sigma) \frac{\|\psi \|_{2,\m}^2}{c_\psi}.
$$
Now by Plancherel formula \eqref{plancherel},
 \begin{equation}\label{eqsigmaeps}
 \int_{\Sigma }\abs{\W_\psi  f (a,x)}^{2} \d\n(a,x)=  1-\int_{\Sigma^c }\abs{\W_\psi  f (a,x)}^{2} \d\n(a,x)\ge  1-\eps.
 \end{equation}
This allows   to conclude.
\end{proof}
In particular, when $\eps = 0$, then \eqref{eqds} implies that the Hankel wavelet transform is concentrated in $\Sigma$ and its support satisfies
\begin{equation}
 \n(\supp \W_\psi  f)\ge \frac{c_\psi}{\|\psi \|_{2,\m}^2}.
\end{equation}
This means that, the support or the  essential support  of the Hankel  wavelet transform cannot be too small.
On the other hand if we take a small subset $\Sigma\subset \r$ such that $\n(\Sigma)< \frac{c_\psi}{\|\psi \|_{2,\m}^2} ,$ then (see \eg \cite[Proposition 3.1]{bhn}), for all $f\in  \L $,
\begin{equation}\label{AMw}
 \int_{\Sigma^c} |\W_\psi f(a,b)|^2\d \n(a,b)\ge \left( 1-\frac{\|\psi \|_{2,\m}^2}{c_\psi} \n(\Sigma)\right)    \|f\|_{2,\m}^2 .
\end{equation}
 In particular if the Hankel wavelet transform is supported in $\Sigma$ (that is  $\supp \W_\psi f\subset \Sigma$), such that $\n(\Sigma)< \frac{c_\psi}{\|\psi \|_{2,\m}^2}$, then $f $ is the zero function.

\begin{theorem}\  \label{procwt}
Let $p\ge 1$ and let $\psi$, $\phi$ two admissible wavelets. Then for all $f,g\in \L$, the function $\W_\psi  f\, \W_\phi  g $
belongs to  $L^p_\alpha(\r)$, and
\begin{equation}\label{eqpsiphinorm0}
  \|\W_\psi  f \, \W_\phi  g  \|_{p,\n}^p\le  \left(\frac{\|\psi \|_{2,\m}\|\phi \|_{2,\m}}{\sqrt{c_\psi c_\phi}}\right)^{p- 1}   \|f\|_{2,\m}^p\|g\|_{2,\m}^p .
\end{equation}
\end{theorem}
\begin{proof}
 By Cauchy-Schwartz's inequality and the Plancherel formula \eqref{plancherel}, we have
 $$
 \int_\r |\W_\psi  f (a,x)  \W_\phi  g (a,x) |\d\n(a,x)\le \|\W_\psi  f\|_{2,\n} \|\W_\phi  g\|_{2,\n}=  \|f\|_{2,\m}\|g\|_{2,\m}.
 $$
This implies that  the function $ \W_\psi f\,\W_\phi  g$ belongs to $L^1_\alpha(\r)$, and
\begin{equation}\label{eqnL1}
\|\W_\psi  f\,\W_\phi  g\|_{1,\n}   \le  \|f\|_{2,\m}\|g\|_{2,\m}.
\end{equation}
On the other hand, from \eqref{norminfty}, we have for $(a,x)\in\r$,
$$
|\W_\psi  f (a,x)  \W_\phi  g (a,x) |\le  \|f\|_{2,\m}\|g\|_{2,\m} \frac{\|\psi \|_{2,\m}\|\phi \|_{2,\m}}{\sqrt{c_\psi c_\phi}}.
$$
Therefore,  the function $ \W_\psi f\,\W_\phi  g$ belongs to $L^\infty_\alpha(\r)$, and
\begin{equation}\label{eqnLinfty}
\|\W_\psi  f\,\W_\phi  g\|_{\infty}   \le  \|f\|_{2,\m}\|g\|_{2,\m} \frac{\|\psi \|_{2,\m}\|\phi \|_{2,\m}}{\sqrt{c_\psi c_\phi}}.
\end{equation}
Thus by \eqref{eqnL1}, \eqref{eqnLinfty} and  by an interpolation theorem we obtain the desired result. 
\end{proof}
\begin{remark}
  Theorem \ref{procwt} implies in particular that $\W_\psi  f\in L^p_\alpha(\r)$,   $p\ge2$, with
 \begin{equation}\label{eqpsiphinorm}
 \|\W_\psi  f\|^{p}_{p,\n}\le \left(\frac{\|\psi \|_{2,\m}^2}{ c_\psi  }\right)^{p/2- 1}   \|f\|_{2,\m}^{p} .
\end{equation}
This result can be considered as a Lieb-type inequality \cite{lieb}.
\end{remark}
\begin{corollary}\
Let   $f\in\l(\eps,\Sigma,\psi)$. Then 
 for every $p>2$,
\begin{equation}\label{eqliebth}
\n(\Sigma)\ge  \frac{c_\psi}{\|\psi \|_{2,\m}^2} \, (1-\eps)^{\frac{p}{p-2}} .
\end{equation}
\end{corollary}

\begin{proof}
Let  $f\in\l(\eps,\Sigma,\psi)$, such that $\|f\|_{2,\m}=1$. Then by H\"{o}lder's inequality,
$$
\int_\Sigma |\W_\psi  f(a,x)|^2\d\n(a,x)\le \n(\Sigma)^{1-2/p} \left(\int_\r |\W_\psi  f(a,x)|^{p}\d\n(a,x)\right)^{2/p}.
$$
Thus by \eqref{eqsigmaeps} and \eqref{eqpsiphinorm}, we have
 $$
1-\eps\le\int_\Sigma |\W_\psi  f(a,x)|^2\d\n(a,x)\le \n(\Sigma)^{1-2/p} \left(\frac{\|\psi \|_{2,\m}^2}{ c_\psi }\right)^{1-2/p}.
 $$
 The proof is complete.
\end{proof}


\section{Logarithmic uncertainty principles for the HWT}

\subsection{On Heisenberg-type uncertainty inequalities for the HWT}
The  Heisenberg-type uncertainty principle for the Hankel tranform states \cite[Theorem 2.1]{JAT}   (see also \cite{bowie, rosler}): For all $f\in  \L$,
\begin{equation}\label{cHpw2}
  \|  x\, f  \|_{2,\m}^2+   \|  \xi \,  \H(f)\|_{2,\m }^2 \ge  (2\alpha+2)   \|  f  \|_{2,\m}^2   ,
\end{equation}
and by a well-known dilation argument (see \cite[Corollary 2.2]{JAT}),  we obtain
\begin{equation}\label{cHpw}
  \|  x\, f  \|_{2,\m}\|  \xi \,  \H(f)\|_{2,\m } \ge (\alpha+1)   \|f\|_{2,\m}^2  .
\end{equation}

Some Heisenberg-type uncertainty inequalities for the HWT were proved in \cite{cyrine}. In this subsection we will   prove the
 following  Heisenberg-type uncertainty inequality, comparing the concentration of $\W_{\psi}f$ in position (with respect to the $x$-variable) and the concentration of $\H(f)$ in frequency (with respect to the $\xi$-variable).

\begin{theorem}\ \label{thH}
For all $f\in  \L $, we have
\begin{equation}\label{cHpw3}
  \| x \,  \W_{\psi}f  \|_{2,\n}^2+   \|    \xi \,  \H(f) \|_{2,\m }^2 \ge (2\alpha+2)    \|  f  \|_{2,\m}^2  ,
\end{equation}
or equivalently,
\begin{equation}\label{cHpw4}
  \| x  \,  \W_{\psi}f  \|_{2,\n} \,  \|\xi \,  \H(f)   \|_{2,\m } \ge (\alpha+1)    \|  f  \|_{2,\m}^2  .
\end{equation}
\end{theorem}

\begin{proof}
It is clear that, if Inequality \eqref{cHpw4} holds then Inequality \eqref{cHpw3} is trivial, since $s^2+t^2\ge 2st$. Therefore it is enough
to prove Inequality \eqref{cHpw3}, and that \eqref{cHpw3} implies \eqref{cHpw4}.

Heisenberg's inequality \eqref{cHpw2} for the function $\W_{\psi}f(a,\cdot) \in  \L $  leads to:
\begin{equation*}
\int_{0}^\infty x^2 |\W_{\psi}f(a,x)|^2\d\m(x)+ \int_{0}^\infty  \xi^2  |\H(\W_{\psi}f(a,\cdot))|^2\d\m(\xi) \ge  (2\alpha+2) \|\W_{\psi}f(a,\cdot) \|_{2,\m}^2.
\end{equation*}
Then by \eqref{cwti},
\begin{equation*}
\int_{0}^\infty x^2   |  \W_{\psi}f(a,x)|^2\d\m(x)+ \int_{0}^\infty|\xi|^2|\H(f)(\xi)|^2  \frac{| \H(\psi) (\xi/a)|^2}{c_\psi\, a^{2\alpha+2}}\d\m(\xi) \ge (2\alpha+2) \|\W_{\psi}f(a,\cdot) \|_{2,\m}^2.
  \end{equation*}
Integrating with respect to $a^{2\alpha+1}\d a$, we obtain from the admissibility condition \eqref{admicondintro} and Plancherel formula \eqref{plancherel},
\begin{equation*}
\int_{0}^\infty \int_{0}^\infty x^2   |  \W_{\psi}f(a,x)|^2\d\n(a,x)+   \int_{0}^\infty \xi^2|\H(f)(\xi)|^2   \d\m(\xi) \ge 2(\alpha+1)    \|f \|_{2,\m}^2  .
\end{equation*}
This prove \eqref{cHpw3}. Now by replacing $f$ by $\D_{\lambda}f$ in the previous inequality, we obtain by \eqref{dt} and \eqref{dtt},
 \begin{equation*}
\int_{0}^\infty \int_{0}^\infty x^2   \left|  \W_{\psi}f\left(\frac{a}{\lambda}, \lambda x\right)\right|^2\d\n(a,x)+
   \lambda^{-2\alpha-2}  \int_{0}^\infty  \xi^2\left|\H(f)\left(\frac{\xi}{\lambda}\right)\right|^2   \d\m(\xi) \ge  2(\alpha+1)\|f \|_{2,\m}^2.
 \end{equation*}
Thus, by a suitable change of variables, we get
\begin{equation*}
\lambda^{-2}   \int_{0}^\infty \int_{0}^\infty x^2   |  \W_{\psi}f(a,x)|^2\d\n(a,x)+  \lambda^{ 2} \int_{0}^\infty  \xi^2|\H(f)(\xi)|^2   \d\m(\xi) \ge  (2\alpha+2)   \|f \|_{2,\m}^2.
\end{equation*}
Minimizing the left-hand side of that inequality over $\lambda > 0$, we obtain
\begin{equation*}
 \left(\int_{0}^\infty\int_{0}^\infty x^2   |  \W_{\psi}f(a,x)|^2\d\n(a,x)\right)^{1/2}  \left(\int_{0}^\infty   \xi^2|\H(f)(\xi)|^2   \d\m(\xi)\right)^{1/2} \ge (\alpha+1)\|f \|_{2,\m}^2.
\end{equation*}
This show \eqref{cHpw4} and the proof is complete.
\end{proof}

\begin{remark}
We have, for all $s\in\R $ (see \eg \cite[Equality (4.9)]{cyrine}),
\begin{equation}\label{mellin}
 \|\xi^s \,  \H(f)   \|_{2,\m }= \sqrt{\frac{c_\psi}{\mathcal{M}\left(|\H(\psi)|^2 \right)(2s)} } \| a^s  \,  \W_{\psi}f  \|_{2,\n},
\end{equation}
where $\mathcal{M}$ is the Mellin transform defined by
\begin{equation}
\mathcal{M}(f)(z)=\int_0^\infty x^{-z} f(x) \frac{\d x}{x}.
\end{equation}
Then by \eqref{cHpw}, \eqref{mellin}, we have the following  Heisenberg-type uncertainty inequality, comparing the concentration of $f$ in position (with respect to the $x$-variable) and  the concentration of $\W_{\psi}f$ in scale   (with respect to the $a$-variable),
\begin{equation}
    \|x \,  f  \|_{2,\m }  \| a   \,  \W_{\psi}f  \|_{2,\n}\ge \sqrt{c_\psi^{-1} \mathcal{M}\left(|\H(\psi)|^2 \right)(2)} \, (\alpha+1)\|f \|_{2,\m}^2,
\end{equation}
and by \eqref{cHpw4}, \eqref{mellin}, we obtain the following Heisenberg-type uncertainty inequality which seems the most natural one, comparing the concentration of $\W_{\psi}f$ in position (with respect to the $x$-variable) and the concentration of $\W_{\psi}f$ in scale (with respect to the $a$-variable).
\begin{equation}\label{eqheips}
  \| a   \,  \W_{\psi}f  \|_{2,\n}  \| x   \,  \W_{\psi}f  \|_{2,\n}\ge \sqrt{c_\psi^{-1} \mathcal{M}\left(|\H(\psi)|^2 \right)(2)}  \, (\alpha+1)\|f \|_{2,\m}^2,
\end{equation}

More generally, by following the same way of Theorem \ref{thH}, and by using \eqref{mellin}, and the following general form of Heisenberg-type uncertainty principle for the Hankel transform,
 \begin{equation}
  \| x^s   \,   f  \|_{2,\m}^\beta  \| x^\beta   \,  \H(f)  \|_{2,\m}^s\ge c(s,\alpha,\beta)     \|f \|_{2,\m}^{s+\beta}\qquad s,\beta>0.
\end{equation}
the author in  \cite[Theorem 4.2]{cyrine} has deduced  that, for any $s,\beta>0$
there exists a constant $c(s,\alpha,\beta)$ such that for all $f\in\L$,
   \begin{equation}\label{heimillen}
  \| a^s   \,  \W_{\psi}f  \|_{2,\n}^\beta  \| x^\beta   \,  \W_{\psi}f  \|_{2,\n}^s\ge c(s,\alpha,\beta) \left(c_\psi^{-1} \mathcal{M}\left(|\H(\psi)|^2 \right)(2s)\right)^{\beta/2}  \|f \|_{2,\m}^{s+\beta}.
\end{equation}

In this paper we are interested in Heisenberg-type uncertainty inequalities like \eqref{eqheips} and \eqref{heimillen} (but without using the Mellin transform), limiting the concentration of $\W_{\psi}f$ in position and scale. A  related results for the windowed Fourier and windowed Hankel transforms were proved in \cite{BDJ, GO}, limiting the concentration in position and frequency.
\end{remark}


\subsection{Pitt-type inequality for the HWT}
The aim of this subsection is to prove an analogue of Pitt's inequality for the Hankel wavelet transform. Recall that,
the  Pitt-type inequality for the Hankel transform (see \cite[Theorem 3.9]{omri}) states: For every $f$ in the Schwartz class $\ss_e(\R)$ and $0\le \beta<  \alpha+1$,
\begin{equation}\label{cpitt}
  \left\| \xi^{-\beta}\H(f) \,\right\|_{2,\m }
    \le C_{\alpha,\beta}   \left\|x^{\beta}f\right\|_{2,\m},
\end{equation}
where 
\begin{equation}
C_{\alpha,\beta} = 2^{-\beta}\frac{\Gamma\left(\frac{\alpha-\beta+1}{2}\right)}{\Gamma\left(\frac{\alpha+\beta+1}{2}\right)} .
\end{equation}
As noted in \cite{omri}, the constant $C_{s,\alpha} $ is the best (smallest) constant for Pitt's inequality \eqref{cpitt}.  Then we derive the following Pitt-type inequality for the Hankel wavelet transform.
\begin{theorem}\
For all function $f\in\ss_e(\R)$ such that $\W_{\psi}f(a,\cdot)\in \ss_e(\R)$, we have
 \begin{equation}  \label{pittcwt}
  \left\|a^{-\beta} \, \W_{\psi}f   \right\|_{2,\n}
 \le  C_{\alpha,\beta}(\psi) \left\|x^{\beta} \, \W_{\psi}f   \right\|_{2,\n}, \qquad 0\le \beta< \alpha+1,
 \end{equation}
 where
  \begin{equation}  \label{pittcwtconstant}
 C_{\alpha,\beta}(\psi) =
\sqrt{ c_\psi^{-1} \mathcal{M}\left(|\H(\psi)|^2\right)(-2\beta) }\,2^{-\beta}
 \frac{\Gamma\left(\frac{\alpha-\beta+1}{2}\right)}{\Gamma\left(\frac{\alpha+\beta+1}{2}\right)} .
  \end{equation}
\end{theorem}

 \begin{proof}
By applying the  Pitt-type inequality \eqref{cpitt} for the function $\W_{\psi}f(a,\cdot)\in \ss_e(\R)$, we obtain
\begin{equation*}
 \int_{0}^\infty \xi^{-2\beta} \abs{\H(\W_{\psi}f(a,\cdot))(\xi)}^2 \d\m(\xi)\le C_{\alpha,\beta}^2 \int_{0}^\infty  x^{2\beta} \abs{\W_{\psi}f(a,x)}^2 \d\m(x).
\end{equation*}
Then, by \eqref{cwti},
\begin{equation*}
 \frac{1}{c_\psi a^{2\alpha+2}} \int_{0}^\infty \xi^{-2\beta}|\H(f)(\xi)|^2  |\H(\psi)(\xi/a) |^2 \d\m(\xi)
 \le C_{\alpha,\beta}^2 \int_{0}^\infty x^{2\beta}\abs{ \W_{\psi}f(a,x)}^2 \d\m(x).
 \end{equation*}
 Integrating  the last inequality with respect to the measure $a^{2\alpha+1} \d a$, we obtain
 \begin{equation*} 
 \frac{1}{c_\psi} \int_{0}^\infty \xi^{-2\beta} |\H(f)(\xi)|^2  \left(\int_0^\infty |\H(\psi)(\xi/a) |^2\frac{\d a}{a}\right)\d\m(\xi)
 \le C_{\alpha,\beta}^2 \int_{\r}x^{2\beta}  \abs{ \W_{\psi}f(a,x)}^2 \d\n(a,x).
 \end{equation*}
 Therefore, since $\psi$ is an admissible wavelet, then by \eqref{admicondintro}, we conclude that
  \begin{equation*} 
  \int_{0}^\infty \xi^{-2\beta} |\H(f)(\xi)|^2   \d\m(\xi)
 \le C_{\alpha,\beta}^2 \int_{\r} x^{2\beta}  \abs{ \W_{\psi}f(a,x)}^2 \d\n(a,x).
 \end{equation*}
 Thus by \eqref{mellin}, we obtain the desired result.
\end{proof}

\begin{remark}\label{rem}
Notice that, if $\beta=0$ then $C_{\alpha,0}=1$  and $\mathcal{M}\left(|\H(\psi)|^2 \right)(0)=c_\psi$, then
  we obtain an equality in Pitt-type inequality \eqref{pittcwt} for any function $f\in\ss_e(\R)$ such that $\W_{\psi}f(a,\cdot)\in \ss_e(\R)$.
\end{remark}

\subsection{Benkner-type uncertainty principle for the HWT}
Now will use  the Pitt-type inequality \eqref{pittcwt} to obtain the following logarithmic uncertainty principle (also known as Beckner-type uncertainty principle) for the HWT.
\begin{theorem}\ \label{thln}
For all function $f\in\ss_e(\R)$ such that $\W_{\psi}f(a,\cdot)\in \ss_e(\R)$, we have
 \begin{equation}  \label{lncwt}
  \int_{\r} \ln(a) \:\abs{\W_{\psi}f(a,x)}^2 \d\n(a,x)+ \int_{\r} \ln(x) \:\abs{\W_{\psi}f(a,x)}^2 \d\n(a,x)  \ge
C_{\alpha}(\psi)  \|f\|^2_{2,\m},
 \end{equation}
 where
\begin{equation}  \label{lncwtconstant}
 C_\alpha(\psi) =\ln(2)+\frac{\Gamma'(\frac{\alpha+1}{2})}{\Gamma(\frac{\alpha+1}{2})} -c_\psi^{-1}\left(\int_0^\infty \frac{\ln(a)}{a}\, \abs{\H(\psi)(a)}^2 \d a\right).
 \end{equation}
\end{theorem}
 \begin{proof}
 Since $\W_{\psi}f(a,\cdot)$ is in the Schwartz class $\ss_e(\R)$, then from \eqref{clupint},
\begin{equation*}
 \int_{0}^\infty \ln(x) \:|\W_{\psi}f(a,x)|^2 \d\m(x)+  \int_{0}^\infty\ln(\xi) \: |\H(\W_{\psi}f(a,\cdot))(\xi)|^2 \d\m(\xi)\ge C_\alpha \|\W_{\psi}f(a,\cdot)\|^2_{2,\m},
\end{equation*}
and by \eqref{cwti}, we have
\begin{equation*}
 \int_{0}^\infty \ln(x) \:|\W_{\psi}f(a,x)|^2 \d\m(x)+  \int_{0}^\infty \ln(\xi) \: |\H(f)(\xi)|^2\frac{|\H(\psi)(\xi/a)|^2}{c_\psi\, a^{2\alpha+2}} \d\m(\xi)\ge C_\alpha \|\W_{\psi}f(a,\cdot)\|^2_{2,\m}.
\end{equation*}
Then by integrating the previous inequality with respect to the measure $a^{2\alpha+1}\d a$, and by \eqref{admicondintro}, \eqref{plancherel}, we obtain 
 \begin{equation}  \label{lncwth0}
  \int_{\r} \ln(x) \:\abs{\W_{\psi}f(a,x)}^2 \d\n(a,x)+  \int_{0}^\infty\ln(\xi) \: |\H(f)(\xi)|^2 \d\m(\xi)\ge 
C_\alpha  \|f\|^2_{2,\m}.
 \end{equation}
 On the other hand, by  \eqref{parseval}, \eqref{admicondintro}, \eqref{plancherel} and \eqref{cwti}
 \begin{eqnarray*}
 \int_{\r} \ln(a) \:\abs{\W_{\psi}f(a,x)}^2 \d\n(a,x)   &=&  \int_0^\infty \ln(a)  \left( \int_0^\infty \abs{\W_{\psi}f(a,x)}^2\d\m(x)\right) a^{2\alpha+1}\d a\\
&=&\int_0^\infty \ln(a)  \left( \int_0^\infty \abs{\H(f)(\xi)}^2\frac{\abs{\H(\psi)(\xi/a)}^2}{c_\psi\, a^{2\alpha+2}} \d\m(\xi)\right) a^{2\alpha+1}\d a\\
&=&c_\psi^{-1}\int_0^\infty \abs{\H(f)(\xi)}^2 \left(\int_0^\infty\ln(a)  \abs{\H(\psi)(\xi/a)}^2 \frac{\d a}{a}\right)\d\m(\xi)\\
&=&c_\psi^{-1}\int_0^\infty \abs{\H(f)(\xi)}^2 \left(\int_0^\infty(\ln(\xi)-\ln(a))  \abs{\H(\psi)(a)}^2 \frac{\d a}{a}\right)\d\m(\xi)\\
&=&\int_0^\infty \ln(\xi)\,\abs{\H(f)(\xi)}^2  \d\m(\xi)- C_\psi \|f\|^2_{2,\m},
 \end{eqnarray*}
 where
 \begin{equation}\label{eqcpsi}
 C_\psi= c_\psi^{-1}\left(\int_0^\infty \frac{\ln(a)}{a}\, \abs{\H(\psi)(a)}^2 \d a\right).
\end{equation}
 This completes the proof.
\end{proof}

 \begin{remark} [Another proof of Theorem \ref{thln}]\
If we define  the function $\Phi$   on $[0,\alpha+1)$ by:
\begin{equation*}
\Phi(\beta)= \int_{\r} a^{-2\beta} \abs{ \W_{\psi}f(a,x)}^2 \d\n(a,x)  - C_{\alpha,\beta}(\psi)^2 \int_{\r} x^{2\beta} \abs{ \W_{\psi}f(a,x)}^2 \d\n(a,x),
\end{equation*}
then
\begin{eqnarray}\label{derphi}
  \Phi'(\beta) &=& -2\int_{\r} a^{-2\beta} \ln(a) \abs{ \W_{\psi}f(a,x)}^2 \d\n(a,x)-2 C_{\alpha,\beta}(\psi)^2 \int_{\r} x^{2\beta}  \ln(x)\: \abs{ \W_{\psi}f(a,x)}^2 \d\n(a,x)  \nonumber\\
    &\qquad& -\frac{\d}{\d \beta} \Big(C_{\alpha,\beta}(\psi)^2\Big)\int_{\r} x^{2\beta} \abs{ \W_{\psi}f(a,x)}^2 \d\n(a,x).
\end{eqnarray}
Moreover,  the Pitt-type inequality  \eqref{pittcwt} implies that $  \Phi(\beta)  \le0$ for every  $\beta\in(0,\alpha+1)$, and by Remark \ref{rem} we have $ \Phi(0)=0$. Therefore $ \Phi'(0^+)  \le0.$  Thus by \eqref{plancherel},
 \begin{equation*}
\int_{\r}  \ln(a)\: \abs{ \W_{\psi}f(a,x)}^2 \d\n(a,x)     +\int_{\r}    \ln(x)\: \abs{ \W_{\psi}f(a,x)}^2 \d\n(a,x)
    \ge -\frac{C}{2}\norm{ f}^2_{2,\m},
\end{equation*}
where
\begin{equation*}
C =\frac{\d}{\d \beta} \Big(C_{\alpha,\beta}(\psi)^2\Big)_{|_{\beta=0}}= -2\left(\ln2+\frac{\Gamma' ( \frac{\alpha+1}{2})}{\Gamma(\frac{\alpha+1}{2} )}-c_\psi^{-1}\left(\int_0^\infty \frac{\ln(a)}{a}\, \abs{\H(\psi)(a)}^2 \d a\right)\right).
\end{equation*}
\end{remark}

\medskip

From the logarithmic uncertainty principle for the HWT \eqref{lncwt}, we can derive a Heisenberg-type uncertainty inequality for functions in $\ss_e(\R)$.
\begin{corollary}
For all $f\in \ss_e(\R)$ such that $\W_{\psi}f(a,\cdot)\in \ss_e(\R)$, we have
\begin{equation}\label{eqHcwt}
 \left\| a  \, \W_{\psi}f \right\|_{2,\n}  \left\| x  \, \W_{\psi}f \right\|_{2,\n} \ge   e^{C_\alpha(\psi)}  \,\|f\|^2_{2,\m}.
\end{equation}
\end{corollary}

\begin{proof}
The logarithmic uncertainty principle for the HWT \eqref{lncwt} can be written as follow:
 \begin{equation} 
  \frac{1}{2}\int_{\r} \ln(a^2)\left(\frac{\abs{\W_{\psi}f(a,x)}^2}{\|f\|^2_{2,\m}} \d\n(a,x)\right)+ \frac{1}{2} \int_{\r} \ln(x^2) \left(\frac{ \abs{\W_{\psi}f(a,x)}^2 }{\|f\|^2_{2,\m}}\d\n(a,x)\right)\ge  C_\alpha (\psi) ,
 \end{equation}
 where by 
 \eqref{plancherel},
  $\left(\frac{\abs{\W_{\psi}f(a,x)}^2}{\|f\|^2_{2,\m}} \d\n(a,x)\right)$ 
   is a probability measure on $(0,\infty)^2$.

Now since the logarithm is a concave function, then  by Jensen's inequality,  we have for all nonzero function $f\in \ss_e(\R)$,
 \begin{equation} 
  \ln\left(\int_{\r} a^2\frac{\abs{\W_{\psi}f(a,x)}^2}{\|f\|^2_{2,\m}} \d\n(a,x)\right)^{1/2}+
     \ln\left( \int_{\r} x^2 \frac{\abs{\W_{\psi}f(a,x)}^2}{\|f\|^2_{2,\m}} \d\n(a,x)\right)^{1/2}\ge  C_\alpha (\psi)  .
 \end{equation}
Then
\begin{equation} 
  \ln\left(\frac{\left\| a \, \W_{\psi}f \right\|_{2,\n}  \left\| x \, \W_{\psi}f \right\|_{2,\n}  }{\|f\|^2_{2,\m}} \right) \ge  C_\alpha (\psi)   .
 \end{equation}
 Therefore
 \begin{equation} 
 \|a \W_{\psi}f  \|_{2,\n}    \|x \W_{\psi}f  \|_{2,\n}  \ge  e^{C_\alpha (\psi) } \,\|f\|^2_{2,\m} .
 \end{equation}
Thus the result follows.
\end{proof}

\begin{remark}
By adapting the proof of the last corollary to the logarithmic uncertainty principle \eqref{clupint}, we will derive a new Heisenberg-type uncertainty inequality for the Hankel transform, that is, for all nonzero function $f$ in the Schwartz space $\ss_e(\R)$, we have

\begin{equation}\label{eqHcwthankelnew}
 \left\| x\,f \right\|_{2,\m}  \left\| \xi  \, \H(f) \right\|_{2,\m} \ge  2\exp\left( \frac{\Gamma'(\frac{\alpha+1}{2})}{\Gamma(\frac{\alpha+1}{2})} \right) \,\|f\|^2_{2,\m}.
\end{equation}
Now Since
\begin{equation} \label{approgamma}
  \frac{\Gamma'(z) }{\Gamma(z)} = \ln z-\frac{1}{2z}-2\int_0^\infty \frac{t}{(t^2+z^2)(e^{2\pi t}-1)}\d t,
\end{equation}
then
\begin{equation}
 2\exp\left( \frac{\Gamma'(\frac{\alpha+1}{2})}{\Gamma(\frac{\alpha+1}{2})} \right)  \approx (\alpha+1), \qquad \mathrm{for} \quad \alpha\gg 1,
\end{equation}
which is the optimal constant in Heisenberg's Inequality \eqref{cHpw}.


\end{remark}


\subsection{Hirschman-Beckner entropic uncertainty inequality   for the HWT}
Following a conjecture by   Hirschman \cite{hirchman}, Beckner \cite{beckner} proved  that for all $f\in L^2(\R^d)$, such that $\|f \|_{2}=1$,
\begin{equation}\label{ineqec}
-\int_{\R^d} |f(x)|^2 \,\ln\big(|f(x)|^2\big)\d x-\int_{\R^d} |\widehat{f}(\xi)|^2 \,\ln\big(|\widehat{f}(\xi)|^2\big)\d\xi
\ge d\ln(e/2).
\end{equation}
Now we define the Shannon's differential entropy by (see \cite{shannon}),
 \begin{equation}\label{entropydef}
 \e(\rho)=-\int_\r \rho(a,x) \,\ln\big(\rho(a,x)\big)\d\n(a,x),
\end{equation}
where $\rho$ is probability density function  on $\r$  satisfying $\|\rho\|_{1,\n}=1$.
The aim of this section is to  prove a Hirschman-Beckner entropic uncertainty inequality, and a Heisenberg-type uncertainty relations for the HWT. 
As a first result, we state the following preliminary lemma.
\begin{lemma}\ \label{lemmaentropy}
For every $ x\in[0,1)$ and every $ p\in(2,3]$, we have
 \begin{equation}\label{eqlnsxle2}
0\le \frac{x^2-x^p}{p-2} \le-x^2\ln x .
\end{equation}
\end{lemma}
\begin{proof}
For $x\in(0,1)$, we define the function
$$
p\mapsto S_x(p)=\frac{x^p-x^2}{p-2},\qquad p\in(2,3].
$$
Then it derivative satisfies
$$
 S_x'(p)=\frac{(p-2)x^p\ln x+x^2-x^p}{(p-2)^2} ,\qquad p\in(2,3].
$$
For every  $x\in(0,1)$, the function $p\mapsto D_x(p)=(p-2)x^p\ln x+x^2-x^p $ is differentiable on $(2,3]$, and
$$
D_x'(p)= (p-2)x^p\ln^2 x\ge 0.
$$
Then for every  $x\in(0,1)$, the function $D_x$ is increasing on $(2,3]$, and $ D_x(2^+) = 0 $. Thus $S_x$ is also an increasing function on $(2,3]$.  In particular, for every  $x\in(0,1)$,
 \begin{equation}
x^2\ln x= \lim_{p\to 2^+} S_x(p)\le S_x(p),\qquad p\in(2,3].
\end{equation}
The last inequality remains true for $x\to 0^+$, and then
 \begin{equation}
0\le \frac{x^2-x^p}{p-2} \le-x^2\ln x ,\qquad x\in[0,1),\quad p\in(2,3].
\end{equation}
This completes the proof of the lemma.
\end{proof}
Now  we will prove the main result of this section.
\begin{theorem}\
Let $\psi$ an   admissible  wavelet such that $\|\psi \|_{2,\m}^2\le c_\psi$. Then
for all  nonzero  function $f\in \L$,  
\begin{equation} \label{ineqE}
-\int_\r |\W_\psi f(a,x)|^2\ln \left(|\W_\psi f(a,x)|^2\right)\d\n(a,x)
\ge \|f \|_{2,\m}^2   \ln \left(\frac{c_\psi}{\|\psi\|_{2,\m}^2\|f \|_{2,\m}^2}\right).
\end{equation}
\end{theorem}
\begin{proof}
Let $f$ be a nonzero function in $\L $,  such that $\|f \|_{2,\m}=1$,
and suppose that $$\e\left( |\W_\psi f|^2 \right)=-\int_\r |\W_\psi f(a,x)|^2\ln \left(|\W_\psi f(a,x)|^2\right)\d\n(a,x)<\infty .$$
Then by   \eqref{norminfty},  we have for all $(a,x)\in\r$, $$ |\W_\psi f(a,x)|  \le   1.$$
 Therefore $$ \e\left(|\W_\psi f|^2 \right)\ge 0, $$ and
from Lemma \ref{lemmaentropy}, we have for every $p\in(2,3]$,
 \begin{equation}\label{eqen}
0\le \frac{ \abs{\W_\psi f(a,x)}^2-   \abs{\W_\psi f(a,x)}^p }{p-2}\le -  \abs{\W_\psi f(a,x)}^2\ln \left( |\W_\psi f(a,x)|\right) .
\end{equation}
Now, from Theorem \ref{procwt}, we know that $\W_\psi f\in L^p_\alpha(\r)$, $p\ge2$,  with
$$
\|\W_\psi f\|^p_{p,\n}\le \left(\frac{\|\psi \|_{2,\m}^2}{c_\psi}  \right)^{p/2-1}.
$$
Then the function $p\mapsto \varphi(p)$ defined on $[2,\infty)$ by
$$
\varphi(p)=\|\W_\psi f\|^p_{p,\n}-\left(\frac{\|\psi \|_{2,\m}^2}{c_\psi}  \right)^{p/2-1}
$$
is negative and by Plancherel formula \eqref{plancherel}, it satisfies $ \varphi(2)=0$.
Thus $ \varphi'(2)\le0$.

On the other hand by \eqref{eqen}, we have for every $p\in(2,3]$, and all $(a,x)\in\r$,
 \begin{eqnarray*}
\int_\r\abs{\frac{ \abs{ \W_\psi f(a,x)}^2-   \abs{ \W_\psi f(a,x)}^p }{p-2}}\d\n(a,x)
&\le&  -  \int_\r\abs{ \W_\psi f(a,x)}^2\ln \left( |\W_\psi f(a,x)|\right)\d\n(a,x), \\
 &=& -\frac{1}{2} \e\left( |\W_\psi f|^2\right)<\infty.
 \end{eqnarray*}
Moreover by Plancherel formula \eqref{plancherel}, we have for every $p>3$, and all $(a,b)\in\r$,
 \begin{eqnarray*}
&&\int_\r\abs{\frac{ \abs{ \W_\psi f(a,x)}^2-   \abs{ \W_\psi f(a,x)}^p }{p-2}}\d\n(a,x)\\
&\le& 2\int_\r\abs{ \W_\psi f(a,x)}^2\d\n(a,x)=2 <\infty.
  \end{eqnarray*}
 Consequently, by the Lebesgue  dominated convergence theorem,
 \begin{eqnarray*}
   \left( \frac{\d}{\d p}  \norm{\W_\psi f}^p_{p,\n}\right)_{|_{p=2^+}}&=&  \int_\r \lim_{p\to 2^+} \frac{\abs{\W_\psi f(a,x)}^p-\abs{ \W_\psi f(a,x)}^2 }{p-2}\d\n(a,x)\\
     &=&  \int_\r   \abs{\W_\psi f(a,x)}^2\ln \left( |\W_\psi f(a,x)|\right)  \d\n(a,x)\\
     &=&  - \frac{1}{2} \e\left( |\W_\psi f|^2\right).
 \end{eqnarray*}
 Therefore
 \begin{eqnarray*}
 \varphi'(2^+)&=& \frac{\d}{\d p} \varphi(p)_{|_{p=2^+}}\\
 &=&- \frac{1}{2} \e\left( |\W_\psi f|^2\right) + \frac{1}{2} \ln\left( \frac{c_\psi}{\|\psi\|_{2,\m}^2}\right)\le0.
 \end{eqnarray*}
Thus
 \begin{equation} \label{eqent0}
  \e\left(  |\W_\psi f|^2\right)\ge    \ln\left( \frac{c_\psi}{\|\psi\|_{2,\m}^2}\right)  .
 \end{equation}
Finally, if $f\in\L$ is any nonzero function, then if we replace $f$ by in  $\frac{f}{\|f \|_{2,\m}}$ in \eqref{eqent0}, we obtain the desired result.
\end{proof}

Consequently we derive the following Heisenberg-type uncertainty inequality for the HWT.
\begin{corollary}\
Let $s,\beta>0$, and let $\psi$   be an admissible  wavelet such that $\|\psi \|_{2,\m}^2\le c_\psi$. Then there exists a positive constant $C_{s,\alpha,\beta}$ such that,
for all nonzero function $f\in \L$,  
\begin{equation} \label{heisum}
\|a^{ s}\, \W_\psi f  \|_{2,\n}^2  + \|x^\beta \,\W_\psi f  \|_{2,\n}^2  \ge C_{s,\alpha,\beta} \, \|f \|_{2,\m}^2,
\end{equation}
where
\begin{equation}\label{heisumconstant}
C_{s,\alpha,\beta}=\frac{(\alpha+1)(s+\beta)}{s\beta}   \exp\left[-1+\frac{ s\beta}{(\alpha+1)(s+\beta)}
\ln \left(\frac{2^{\alpha+2} s\beta\,\Gamma(\alpha+1) }{\Gamma(\frac{\alpha+1}{s} )\Gamma(\frac{\alpha+1}{\beta})}  \frac{c_\psi}{ \|\psi\|_{2,\m}^2}\right)  \right].
\end{equation}

\end{corollary}
\begin{proof}
Let  $f\in \L $,   such that $\|f \|_{2,\m}=1$, and let $t,s,\beta>0$. Then a straightforward computation gives,
$$
\int_0^\infty\int_0^\infty e^{-\frac{a^{2s}+x^{2\beta}}{t}}\d\n(a,x)= c(s,\alpha,\beta)\;t^{\frac{\alpha+1}{\beta}+\frac{\alpha+1}{s}},
$$
where
$$
c(s,\alpha,\beta)= \frac{\Gamma(\frac{\alpha+1}{s})\Gamma( \frac{\alpha+1}{\beta})}{  2^{\alpha+2}  s\beta\,\Gamma(\alpha+1)  }.
$$
Now let $Q_{t,s,\alpha,\beta}$ the function defined on $\r$  by
$$
Q_{t,s,\alpha,\beta}(a,x)= \frac{e^{-\frac{a^{2s}+x^{2\beta}}{t}}}{c(s,\alpha,\beta)\;t^{\frac{\alpha+1}{\beta}+\frac{\alpha+1}{s}}}.
$$
Then $Q_{t,s,\alpha,\beta}(a,x)\d\n(a,x)$   is a probability measure on $\r$, and since the function $t\mapsto t\ln t$ is convex, we obtain by  Jensen's inequality,
$$
\int_\r |\W_\psi f(a,x)|^2 \ln \left(\frac{|\W_\psi f(a,x)|^2}{Q_{t,s,\alpha,\beta}(a,x)} \right)\d\n(a,x)\ge 0.
$$
Therefore
$$
\frac{1}{t}\left( \|a^{ s} \,\W_\psi f  \|_{2,\n}^2  + \|x^\beta\, \W_\psi f  \|_{2,\n}^2\right) \ge \e\left(|\W_\psi f |^2\right)-\ln (t^{\frac{\alpha+1}{\beta}+\frac{\alpha+1}{s}})- \ln c(s,\alpha,\beta).
$$
It follows that by \eqref{eqent0},
$$
 \|a^{ s}\, \W_\psi f  \|_{2,\n}^2  + \|x^\beta\, \W_\psi f  \|_{2,\n}^2  \ge t\left[\ln \left(\frac{c_\psi}{\|\psi\|_{2,\m}^2}\right) -\ln (t^{\frac{\alpha+1}{\beta}+\frac{\alpha+1}{s}})- \ln c(s,\alpha,\beta)\right].
$$
  Minimizing the right hand side with $t_0=e^{\frac{ s\beta c}{(\alpha+1)(s+\beta)}-1} $, we obtain
  $$
 \|a^{s}\, \W_\psi f  \|_{2,\n}^2  + \|x^\beta\, \W_\psi f  \|_{2,\n}^2  \ge C_{s,\alpha,\beta}
$$
 where $C_{s,\alpha,\beta}=\frac{(\alpha+1)(s+\beta)}{ s\beta}t_0$, and $c=\ln \left(\frac{c_\psi}{ c(s,\alpha,\beta)\,\|\psi\|_{2,\m}^2}\right) $.
 Finally replacing $f$ by $\frac{f}{\|f\|_{2,\m}}$, we conclude the desired result,
  \end{proof}
 Moreover we have the following Heisenberg-type uncertainty like \eqref{heimillen}, without involving the Mellin transform.
 \begin{corollary}\
Let $s,\beta>0$, and let $\psi$ an be an admissible  wavelet such that $\|\psi \|_{2,\m}^2\le c_\psi$. Then there exists a positive constant $C(s,\alpha,\beta)$ such that, for all nonzero function $f\in \L$,  
\begin{equation} \label{heipro}
\|a^{ s} \, \W_\psi f  \|_{2,\n}^\beta  \|x^\beta \,\W_\psi f  \|_{2,\n}^s  \ge C(s,\alpha,\beta) \, \|f \|_{2,\m}^{s+\beta},
\end{equation}
where
\begin{equation} \label{heiproconstant}
C(s,\alpha,\beta) =\left(\frac{s}{\beta} \right)^{\frac{s}{2}} \left(\frac{\beta}{s+\beta}\,C_{s,\alpha,\beta}\right)^{\frac{s+\beta}{2}}.
\end{equation}
\end{corollary}

\begin{proof}
Replacing $f$ by $\D_\lambda f$ in Inequality \eqref{heisum}, we obtain by \eqref{dtt},
$$
\int_\r a^{2s}  \left|\W_\psi f\left(\frac{a}{\lambda},\lambda x\right) \right|^2\d\n(a,x)
+ \int_\r x^{2\beta}  \left|\W_\psi f\left(\frac{a}{\lambda},\lambda x\right) \right|^2\d\n(a,x)\ge C_{s,\alpha,\beta}\|f\|_{2,\m}^2.
$$
Then a suitable change of variables gives
$$
\lambda^{ 2s}\|a^{ s} \, \W_\psi f  \|_{2,\n}^2
+\lambda^{-2\beta} \|x^\beta \,\W_\psi f  \|_{2,\n}^2\ge C_{s,\alpha,\beta}\|f\|_{2,\m}^2.
$$
Minimizing the left hand side with $\lambda=\left(\frac{\beta\|x^\beta \,\W_\psi f  \|_{2,\n}^2}{s\|a^{s} \, \W_\psi f  \|_{2,\n}^2 }\right)^{\frac{1}{2(s+\beta)}}$,
we obtain
$$
\frac{s+\beta}{\beta}\left(\frac{s}{\beta}\right)^{\frac{-s}{s+\beta}}
\|a^{-s} \, \W_\psi f  \|_{2,\n}^{\frac{2\beta}{s+\beta}}
\|x^\beta \,\W_\psi f  \|_{2,\n}^{\frac{2s}{s+\beta}}\ge C_{s,\alpha,\beta}\|f\|_{2,\m}^{2},
$$
which allows to conclude.
\end{proof}

\end{document}